\documentclass[11pt,a4paper]{article}

\usepackage[T1]{fontenc}
\usepackage[utf8]{inputenc}
\usepackage{lmodern}
\usepackage[breaklinks]{hyperref}
\usepackage{amsmath}
\usepackage{mathtools}
\usepackage{amsfonts}
\usepackage{amssymb}
\usepackage{amsthm}
\usepackage{enumitem}
\usepackage{tikz}
\usetikzlibrary{angles,quotes,arrows.meta,calc,intersections}

\theoremstyle{definition}
\newtheorem{thm}{Theorem}[section]
\newtheorem{lem}[thm]{Lemma}

\newtheorem{cor}[thm]{Corollary}

\theoremstyle{remark}
\newtheorem{rem}[thm]{Remark}

\numberwithin{equation}{section}

\newcommand{\MG}{\underset{\sim}{\mathcal{M}}}
\newcommand{\PerpMark}[5]{%
  \coordinate (PM) at ($(#1)!#3!(#2)$);
  \draw[#5]
    ($(PM)!#4!90:(#2)$) --
    ($(PM)!#4!-90:(#2)$);
}
\title{On the Reconstruction of SAS from Other Triangle Congruence Criteria \\
\large Part II: Eliminating the Pons Asinorum}
\author{Roberto Volpe}
\date{September 18, 2026}

\begin{document}
\maketitle

\begin{abstract}
In the first part of this work we showed that, within a Hilbert plane deprived of the
Side-Angle-Side axiom, the Side-Angle-Angle criterion, together with a ray
correspondence principle [\textbf{RCT}], the existence of angle bisectors [\textbf{AB}], the
congruence of supplements of congruent angles [\textbf{SA}], and the Pons Asinorum
[\textbf{PA}], suffices to reconstruct SAS. We left open the question of whether
[\textbf{PA}] is genuinely required alongside the other three principles, noting only a
qualitative asymmetry in the nature of the principles involved. In this second part we
answer this question: we show that
\begin{equation*}
\textrm{SAA},\ [\textbf{RCT}],\ [\textbf{AB}] \;\vdash\; [\textbf{PA}],
\end{equation*}
so that [\textbf{PA}] is redundant among the hypotheses of our main theorem,
which improves to
\begin{equation*}
\textrm{SAA},\ [\textbf{RCT}],\ [\textbf{AB}],\ [\textbf{SA}] \;\vdash\; \textrm{SAS}.
\end{equation*}
The proof adapts an argument recently given by Donnelly, who reconstructs SAS from SAA
together with an angle addition axiom and the existence of angle bisectors.
\end{abstract}

\section{Introduction}\label{PEN2_s1}

In the first part of this work \cite{Volpe2026I} we investigated, within the axiomatic
system $\MG^{-}$ (a Hilbert plane deprived of the SAS axiom [\textbf{C}6]), which combinations of
the classical triangle congruence criteria and auxiliary principles suffice to reconstruct
SAS. Among our results, we established that
\begin{equation}\label{PEN2_ESAA_OLD}
\textrm{SAA},\ [\textbf{RCT}],\ [\textbf{AB}],\ [\textbf{SA}],\ [\textbf{PA}] \;\vdash\; \textrm{SAS}
\end{equation}
in $\MG^{-}$, where [\textbf{RCT}] denotes a ray correspondence principle corresponding to
Theorem 13 of Hilbert's \emph{Grundlagen der Geometrie} \cite{Hilbert1950}, and [\textbf{PA}] denotes the Pons
Asinorum.

In the concluding section of that work we observed a structural asymmetry between the two
reconstructions obtained there, from SSS and from SAA respectively, and we noted explicitly
that we could not settle the question of whether [\textbf{PA}] was truly indispensable
alongside the other hypotheses in \eqref{PEN2_ESAA_OLD}:
\begin{quote}
\emph{``Although we cannot rule out the existence in $\MG^{-}$ of the two synthetic proofs}
SSS\emph{,} [\textbf{RCT}], [\textbf{MS}] \emph{$\vdash$} [\textbf{HA}]\emph{, and}
SAA\emph{,} [\textbf{RCT}], [\textbf{AB}], [\textbf{SA}] \emph{$\vdash$}
[\textbf{PA}]\emph{, the different nature of the principles involved can provide a
qualitative indication of their deductive strength [...] while not allowing us to establish
a formal relation of greater or lesser deductive power.''}
\end{quote}
The present paper settles this question for the second of the two conjectural proofs: we
show that [\textbf{PA}] does follow synthetically from SAA, [\textbf{RCT}], and [\textbf{AB}]
alone, without appeal to continuity or to any principle beyond those already available in
$\MG^{-}$.

The route to this result came from an unplanned direction. After the preprint of
\cite{Volpe2026I} had already been submitted, we became aware of a recent paper by Donnelly
\cite{Donnelly2025}, published shortly before, which addresses a closely related question in
a different axiomatic framework: starting from $\MG^{-}$, Donnelly removes SAS and replaces
it with three new primitive axioms -- Side-Angle-Angle (N1), an angle addition axiom (N2),
and the existence of angle bisectors (N3) -- and reconstructs SAS from these. Among the
intermediate results of that reconstruction, Donnelly proves the Pons Asinorum from N1, N2,
and N3.

A comparison between the two reconstructions is instructive. Both start from the same
criterion, SAA, and share the assumption of angle bisectors ([\textbf{AB}] and N3 coincide).
The difference lies in the second ingredient: where our own reconstruction uses
[\textbf{RCT}], the reconstruction of \cite{Donnelly2025} uses the angle addition axiom N2.
As we detail in Section~\ref{PEN2_s2}, N2 is a substantially stronger assumption than
[\textbf{RCT}]: on the shared basis of SAA and [\textbf{AB}], N2 already implies
[\textbf{PA}] (this is essentially the result of \cite{Donnelly2025} itself), while whether
[\textbf{RCT}] does so was, until now, an open question -- precisely the one left open in
\cite{Volpe2026I}.

By tracing the argument of \cite{Donnelly2025} for the Pons Asinorum back through its
dependencies -- an exterior angle theorem, an alternate interior angle theorem, and a
crossing-angles configuration -- we found that every use of N2 in these three preliminary
lemmas is a use of the ordering theory of angles that N2 generates, and nothing more. Since
[\textbf{RCT}] already suffices, within $\MG^{-}$, to establish the same ordering theory
(trichotomy and transitivity of the angle order relation, established in
\cite{Volpe2026I}), each of these three lemmas can be re-proved from SAA and [\textbf{RCT}]
alone, without N2. The final step to [\textbf{PA}] itself, however, required more than a
direct substitution: \cite{Donnelly2025} leaves the corresponding argument largely
unjustified (see Section~\ref{PEN2_s2}), and adapting it to use only [\textbf{RCT}] led us to
a substantially restructured construction, detailed in Section~\ref{PEN2_s3}. The result is
nonetheless the one announced above: [\textbf{PA}] follows from SAA, [\textbf{RCT}], and
[\textbf{AB}] alone, with no appeal to N2 anywhere in the argument.

The remainder of this paper is organized as follows. Section~\ref{PEN2_s2} places this
result in relation to Donnelly's reconstruction, in particular contrasting the deductive
strength of [\textbf{RCT}] and N2 on the common basis of SAA and [\textbf{AB}], and to Donnelly's
earlier work on the same triangle congruence criteria in the metric setting of Birkhoff.
Section~\ref{PEN2_s3} contains the main results: the exterior angle theorem in its general
form, the alternate interior angle theorem, the crossing-angles lemma, and finally
[\textbf{PA}] itself, each derived from SAA and [\textbf{RCT}] alone (with [\textbf{AB}] entering
only in the last step). Section~\ref{PEN2_s4} draws the consequence for our main theorem and
updates the deductive equivalence of \cite{Volpe2026I} accordingly. We close with some
remarks on what remains open.

\section{Comparison with an Alternative Reconstruction from SAA}\label{PEN2_s2}

Donnelly's reconstruction \cite{Donnelly2025} takes place in the same base system
$\MG^{-}$, to which three new primitive axioms are added:
\begin{equation*}
\textrm{N1}=\textrm{SAA},\qquad
\textrm{N2}=\text{angle addition},\qquad
\textrm{N3}=\text{existence of angle bisectors}.
\end{equation*}
The main result there is
\begin{equation*}
\textrm{N1},\;\textrm{N2},\;\textrm{N3}\;\vdash\;\textrm{SAS},
\end{equation*}
obtained through a chain of intermediate lemmas that includes, in order: an ordering theory
for angles built from N2 alone; the existence of midpoints from N1 and N3; the Pons Asinorum
and its converse from N1, N2, and N3; the congruence of all right angles from N2 alone
(without N1 or N3); and finally SAS itself via a reflection argument.

On the shared basis of SAA and [\textbf{AB}] (which coincides with N1 and N3), the comparison
between [\textbf{RCT}] and N2 is therefore the crux of the matter. We record here what is
established, and what remains open, about their relative strength.

\begin{itemize}
\item N2 implies the ordering theory of angles (trichotomy, transitivity) that Donnelly uses
throughout his reconstruction. So does [\textbf{RCT}], within $\MG^{-}$, as established in
\cite{Volpe2026I}.
\item On the basis of N1 and N3, N2 implies [\textbf{PA}] \cite[Lemmas~17--18]{Donnelly2025}.
We show in Section~\ref{PEN2_s3} that, given SAA and [\textbf{AB}], [\textbf{RCT}] also
suffices to derive [\textbf{PA}].
\item On its own (without N1 or N3), N2 implies the congruence of all right angles
\cite[Lemma~19]{Donnelly2025}.
\end{itemize}

The specific route to [\textbf{PA}] in \cite{Donnelly2025} follows a citation trail that we
make fully explicit here. Lemmas 17 and 18 there (the converse and direct forms of the Pons
Asinorum) are stated with proofs said to be ``similar to the proofs given in
[\textbf{18}]'' -- referring to Donnelly's earlier paper \cite{Donnelly2010}, in the
continuous, Birkhoff-style setting -- without further detail. Tracing this reference, the
argument is given in full in \cite{Donnelly2010}: Lemma 5 there for the converse, and Lemma 6
for the Pons Asinorum itself, the latter built on a crossing-configuration argument (Lemma 4
there) essentially identical to our Theorem~\ref{PEN2_CA} below. Section~\ref{PEN2_s3}
reproduces this argument step by step, synthetically and within $\MG^{-}$, and replaces every
use of angle addition in it -- present in \cite{Donnelly2010} through the Protractor
Postulate, and inherited by \cite{Donnelly2025} through the citation just discussed -- with
[\textbf{RCT}] alone.

\begin{rem}
We note, finally, a point of exposition rather than of substance in \cite{Donnelly2025}. The
proof there of SAS (via reflections) relies on properties of reflections established by
appeal to Donnelly's earlier paper \cite{Donnelly2010}, stated to transfer ``without the use
of continuity''. Tracing this claim, the relevant argument in \cite{Donnelly2010} applies
SAA to two triangles by asserting, among other things, that two right angles located at
different points are congruent to each other -- a step left unjustified there, since in that
continuous setting any two right angles are congruent by definition. In the non-continuous
setting of \cite{Donnelly2025}, this same step requires a genuine theorem -- the congruence
of all right angles, \cite[Lemma~19]{Donnelly2025}. The overall argument is not circular,
since that lemma is established earlier in the paper and is therefore available at the point
where it is needed; but the text does not make explicit that this is precisely where it is
invoked, leaving the correspondence for the reader to reconstruct.
\end{rem}

\section{From SAA and RCT to the Pons Asinorum}\label{PEN2_s3}

We now give the reconstruction announced in the introduction. Throughout this section we work
in $\MG^{-}$ and assume SAA; [\textbf{AB}] is invoked only in Theorem~\ref{PEN2_PA} below.

\begin{thm}[\emph{Exterior angle theorem, general form}]\label{PEN2_EAT}
If we assume SAA and [\textbf{RCT}], then: given a triangle $\triangle ABC$ and a point $D$ such
that $A-C-D$, the exterior angle $\angle BCD$ is greater than each of the two remote interior
angles, that is, $\angle BCD>\angle BAC$ and $\angle BCD>\angle ABC$.
\end{thm}

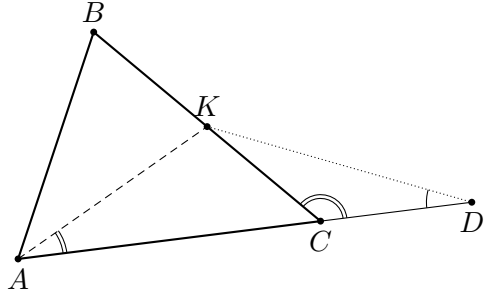
\begin{figure}[ht!]
\centering
\begin{tikzpicture}[scale=1.0,thin]
\coordinate[label=below:$A$] (A) at (0,0);
\coordinate[label=above:$B$] (B) at (1,3);
\coordinate[label=below:$C$] (C) at (4,0.5);
\coordinate[label=below:$D$] (D) at (6,0.75);
\coordinate[label=above:$K$] (K) at (2.5,1.75);

\draw[thick] (A)--(B)--(C)--cycle;
\draw (C)--(D);
\draw[densely dashed] (A)--(K);
\draw[densely dotted] (K)--(D);

\fill (A) circle (1.3pt);
\fill (B) circle (1.3pt);
\fill (C) circle (1.3pt);
\fill (D) circle (1.3pt);
\fill (K) circle (1.3pt);

\pic[draw,thin,angle radius=0.6cm] {angle=K--D--A};
\pic[draw,thin,angle radius=0.6cm] {angle=C--A--K};
\pic[draw,thin, angle radius=0.65cm] {angle=C--A--K};
\pic[draw,thin, angle radius=0.3cm] {angle=D--C--B};
\pic[draw,thin, angle radius=0.35cm] {angle=D--C--B};
\end{tikzpicture}
\caption{Theorem~\ref{PEN2_EAT} Schematic}
\label{PEN2_EAT_f}
\end{figure}

\begin{proof}
Given $\triangle ABC$ and $D$ such that $A-C-D$, suppose for contradiction that
$\angle BAC\geq\angle BCD$.

If $\angle BAC\equiv\angle BCD$, set $K=B$. If instead $\angle BAC>\angle BCD$, by [\textbf{C}4]
and Crossbar let $K$ be the unique point in $\mathrm{int}(BC)$ such that
$\angle KAC\equiv\angle BCD$.

Consider $\triangle KDA$ and $\triangle KDC$, with common side $KD$. Since $A-C-D$, the rays
$\overrightarrow{DA}$ and $\overrightarrow{DC}$ coincide, so trivially
$\angle KDA\equiv\angle KDC$. Since $K\in\mathrm{int}(BC)$ (or $K=B$),
$\overrightarrow{CK}=\overrightarrow{CB}$, hence $\angle KCD=\angle BCD$; and since
$\overrightarrow{AD}=\overrightarrow{AC}$, $\angle KAD=\angle KAC$. By construction
$\angle KAC\equiv\angle BCD$, hence $\angle KAD\equiv\angle KCD$.

The triangles $\triangle KDA$ and $\triangle KDC$ therefore share side $KD$, an angle at $D$,
and an angle at $A$/$C$: by SAA, $\triangle KDA\equiv\triangle KDC$, whence $DA\equiv DC$. But
$A-C-D$ gives $DA=DC+CA>DC$, a contradiction.

Hence $\angle BAC\geq\angle BCD$ is excluded, and by [\textbf{RCT}] (trichotomy of angles) we
conclude $\angle BCD>\angle BAC$.

A symmetric argument, taking $F$ such that $B-C-F$, gives $\angle ACF>\angle ABC$; since
$\angle BCD$ and $\angle ACF$ are vertical angles, $\angle BCD\equiv\angle ACF>\angle ABC$.
\end{proof}

\begin{rem}
Theorem~\ref{PEN2_EAT} is the general form of the exterior angle theorem, valid for any
triangle under the sole hypotheses SAA and [\textbf{RCT}]. In \cite{Volpe2026I} we developed a
special case of the exterior angle theorem for right triangles, which served as an
intermediate result (angle acuteness) used to derive SAS from SSS when SAA was not
available. Here we need the general form, from SAA and [\textbf{RCT}] directly, since the
present context is the SAA branch itself.
\end{rem}

\begin{thm}[\emph{Alternate interior angles}]\label{PEN2_AI}
If we assume SAA and [\textbf{RCT}], then: if two lines are cut by a transversal at two distinct
points so that a pair of alternate interior angles is congruent, the two lines are parallel.
\end{thm}

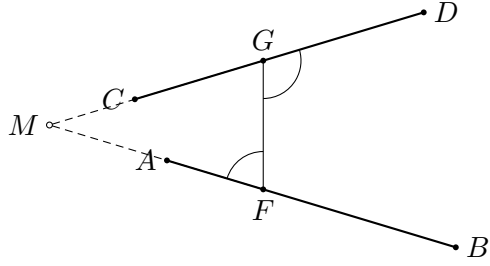
\begin{figure}[ht!]
\centering
\begin{tikzpicture}[scale=0.85,thin]
\coordinate[label=left:$A$] (A) at (-1.5,0.45);
\coordinate[label=right:$B$] (B) at (3,-0.9);
\coordinate[label=left:$C$] (C) at (-2,1.4);
\coordinate[label=right:$D$] (D) at (2.5,2.75);
\coordinate[label=below:$F$] (F) at (0,0);
\coordinate[label=above:$G$] (G) at (0,2);
\coordinate[label=left:$M$] (M) at (-3.33,1.0);

\draw[thick] (A)--(B);
\draw[thick] (C)--(D);
\draw (F)--(G);
\draw[densely dashed] (F)--(M);
\draw[densely dashed] (G)--(M);

\fill (A) circle (1.3pt);
\fill (B) circle (1.3pt);
\fill (C) circle (1.3pt);
\fill (D) circle (1.3pt);
\fill (F) circle (1.3pt);
\fill (G) circle (1.3pt);
\draw (M) circle (1.3pt);

\pic[draw,angle radius=0.5cm] {angle=G--F--A};
\pic[draw,angle radius=0.5cm] {angle=F--G--D};
\end{tikzpicture}
\caption{Theorem~\ref{PEN2_AI} Schematic: the dashed lines show the hypothetical
intersection $M$ used in the proof by contradiction.}
\label{PEN2_AI_f}
\end{figure}

\begin{proof}
Fix notation: let $\overleftrightarrow{AB}$ and $\overleftrightarrow{CD}$ be the two lines,
and $\overleftrightarrow{FG}$ the transversal, distinct from both, with
$F=\overleftrightarrow{AB}\cap\overleftrightarrow{FG}$ and
$G=\overleftrightarrow{CD}\cap\overleftrightarrow{FG}$ (necessarily $F\neq G$, since
$\overleftrightarrow{FG}$ is by definition the line determined by two distinct points, by
[\textbf{I}1]), with $A-F-B$, $C-G-D$, and $A,C$ on the same side of $\overleftrightarrow{FG}$.
Suppose the congruent alternate interior angles are $\angle AFG\equiv\angle DGF$.

Suppose for contradiction that $\overleftrightarrow{AB}$ and $\overleftrightarrow{CD}$ are
not parallel, i.e.\ $\overleftrightarrow{AB}\cap\overleftrightarrow{CD}=M$ for some point $M$.

\emph{Case $M\in\overleftrightarrow{FG}$.} Then $M$ is common to $\overleftrightarrow{AB}$
and $\overleftrightarrow{FG}$, distinct lines: by [\textbf{I}1], $M=F$. Likewise $M=G$. But
$F\neq G$: a contradiction.

\emph{Case $M\notin\overleftrightarrow{FG}$.} Without loss of generality, suppose $M$ is on
the same side of $\overleftrightarrow{FG}$ as $A$ and $C$ (the opposite case is entirely
symmetric, exchanging $A$ with $B$ and $C$ with $D$).

Consider $\triangle FMG$. Since $C-G-D$, $D$ lies on the ray opposite
$\overrightarrow{GC}$ from $G$; thus $\angle FGD$ is the exterior angle of $\triangle FMG$ at
$G$, with remote interior angle $\angle MFG$. Since $M$ is on the same side of
$\overleftrightarrow{FG}$ as $A$, and $M\in\overleftrightarrow{AB}$, we have
$\overrightarrow{FM}=\overrightarrow{FA}$, hence $\angle MFG=\angle AFG$.

By Theorem~\ref{PEN2_EAT}, $\angle FGD>\angle MFG=\angle AFG$. But
$\angle FGD=\angle DGF\equiv\angle AFG$ by hypothesis, contradicting
$\angle FGD>\angle AFG$ (by trichotomy, from [\textbf{RCT}], these cannot both hold).

In either case we reach a contradiction. Hence
$\overleftrightarrow{AB}\cap\overleftrightarrow{CD}=\emptyset$, i.e.\
$\overleftrightarrow{AB}\parallel\overleftrightarrow{CD}$.
\end{proof}

\begin{thm}[Crossing-angles lemma]\label{PEN2_CA}
If we assume SAA and [\textbf{RCT}], then: let $C,D$ be points on opposite sides of
$\overleftrightarrow{AB}$ such that $\angle ABC\equiv\angle BAD$, and let
$G=\overleftrightarrow{AB}\cap CD$. Then $A-G-B$. If moreover $AD\equiv BC$, then
$\triangle DAG\equiv\triangle CBG$ and $G$ is the midpoint of both $AB$ and $CD$.
\end{thm}


\begin{figure}[ht!]
\centering
\begin{tikzpicture}[scale=0.9,thin]
\coordinate[label=left:$A$] (A) at (0,0);
\coordinate[label=right:$B$] (B) at (7,0);
\coordinate[label=above:$C$] (C) at (6,2.3);
\coordinate[label=below:$D$] (D) at (0.6,-2.3);

\draw[thick,name path=AB] (A)--(B);
\draw[densely dashed,name path=CD] (C)--(D);
\path[name intersections={of=AB and CD, by=G}];
\node[label=above:$G$] at (G) {};

\draw (A)--(D);
\draw (B)--(C);

\fill (A) circle (1.3pt);
\fill (B) circle (1.3pt);
\fill (G) circle (1.3pt);
\fill (C) circle (1.3pt);
\fill (D) circle (1.3pt);

\pic[draw,angle radius=0.5cm] {angle=D--A--B};
\pic[draw,angle radius=0.6cm] {angle=D--A--B};
\pic[draw,angle radius=0.5cm] {angle=C--B--A};
\pic[draw,angle radius=0.6cm] {angle=C--B--A};
\pic[draw,angle radius=0.3cm] {angle=A--G--D};
\pic[draw,angle radius=0.3cm] {angle=B--G--C};

\draw ($(A)!0.5!(D)+(-0.10,0.07)$) -- ($(A)!0.5!(D)+(0.10,-0.07)$);
\draw ($(B)!0.5!(C)+(-0.10,0.07)$) -- ($(B)!0.5!(C)+(0.10,-0.07)$);
\end{tikzpicture}
\caption{Theorem~\ref{PEN2_CA} Schematic}
\label{PEN2_CA_f}
\end{figure}
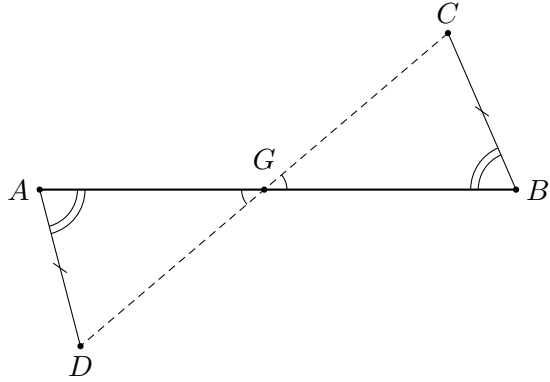

\begin{proof}
\emph{Existence of $G$.} Since $C,D$ are on opposite sides of $\overleftrightarrow{AB}$, by
definition of opposite sides \cite[Definition p. 77]{Greenberg1993} the segment $CD$ meets
$\overleftrightarrow{AB}$ at a point $G$, and by construction $C-G-D$.

\emph{Exclusion of $G=A$.} If $G=A$, then $\angle BAD$ would be the exterior angle of
$\triangle ABC$ at $A$ with remote interior angle $\angle ABC$; by Theorem~\ref{PEN2_EAT},
$\angle BAD>\angle ABC$, contradicting $\angle ABC\equiv\angle BAD$. Hence $G\neq A$.

\emph{Exclusion of $G=B$.} Symmetrically, $G\neq B$.

\emph{Exclusion of $A-B-G$.} Suppose $A-B-G$, so $B\in\mathrm{int}(AG)$. By Pasch applied to
$\triangle AGD$, the line $\overleftrightarrow{CB}$, through the interior point $B$ of side
$AG$, must meet another side, $AD$ or $GD$. Since $\angle ABC\equiv\angle BAD$ are alternate
interior angles for the transversal $\overleftrightarrow{AB}$ cutting
$\overleftrightarrow{BC}$ and $\overleftrightarrow{AD}$ at the distinct points $B,A$,
Theorem~\ref{PEN2_AI} gives $\overleftrightarrow{BC}\parallel\overleftrightarrow{AD}$. So
$\overleftrightarrow{CB}$ meets $GD$ at a point $N$ with $G-N-D$.

Since $C-G-D$, $C\notin\mathrm{int}(GD)$, so $C\neq N$. But $C$ is also common to
$\overleftrightarrow{BC}$ and $\overleftrightarrow{GD}=\overleftrightarrow{CD}$; two distinct
common points force $\overleftrightarrow{BC}=\overleftrightarrow{CD}$, so
$D\in\overleftrightarrow{BC}$ -- contradicting $\overleftrightarrow{BC}\parallel
\overleftrightarrow{AD}$ together with $D\in\overleftrightarrow{AD}$ (as $A,B,C$ are not
collinear). Hence $A-B-G$ is excluded.

\emph{Exclusion of $G-A-B$.} Symmetric argument.

\emph{Conclusion.} With $G=A$, $G=B$, $A-B-G$, $G-A-B$ all excluded, and by [\textbf{O}3]
(trichotomy of betweenness), $A-G-B$.

\emph{Second part.} Assume $AD\equiv BC$. Since $A-G-B$ and $C-G-D$, $\angle AGD$ and
$\angle CGB$ are vertical angles, hence congruent. Since $G\in\mathrm{int}(AB)$,
$\overrightarrow{AG}=\overrightarrow{AB}$, so $\angle DAG=\angle BAD$; and since
$G\in\mathrm{int}(CD)$, $\overrightarrow{BG}=\overrightarrow{BA}$, so $\angle CBG=\angle ABC$.
By hypothesis $\angle ABC\equiv\angle BAD$, hence $\angle DAG\equiv\angle CBG$.

The triangles $\triangle DAG,\triangle CBG$ share: side $DA\equiv CB$, an angle at $G$
(vertical angles), and an angle at $A$/$B$ (just shown): by SAA,
$\triangle DAG\equiv\triangle CBG$. Hence $AG\equiv BG$, $DG\equiv CG$, so $G$ is the
midpoint of both $AB$ and $CD$.
\end{proof}
\begin{lem}[\emph{Converse Pons Asinorum}]\label{PEN2_PA_CONV}
If we assume [\textbf{AB}], and $\angle BAC\equiv\angle BCA$, then $BA\equiv BC$.
\end{lem}
\begin{proof}
By [\textbf{AB}], let $K\in\mathrm{int}(AC)$ be the unique point such that
$\overrightarrow{BK}$ bisects $\angle ABC$; hence $\angle ABK\equiv\angle CBK$.

Since $K\in\mathrm{int}(AC)$, $\overrightarrow{AK}=\overrightarrow{AC}$ and
$\overrightarrow{CK}=\overrightarrow{CA}$, so $\angle BAK=\angle BAC$ and
$\angle BCK=\angle BCA$. By hypothesis $\angle BAC\equiv\angle BCA$, hence
\begin{equation*}
\angle BAK\equiv\angle BCK.
\end{equation*}

Consider $\triangle KBA$ and $\triangle KBC$: they share side $BK$, have
$\angle ABK\equiv\angle CBK$ (the angle at $B$), and $\angle BAK\equiv\angle BCK$ (the angle
at $A$/$C$, just shown). By SAA -- side $BK$, angle at $B$, angle at $A$/$C$ -- we obtain
$\triangle KBA\equiv\triangle KBC$, hence $BA\equiv BC$.
\end{proof}

\begin{thm}\label{PEN2_PA}
If we assume SAA, [\textbf{RCT}], and [\textbf{AB}], then [\textbf{PA}] follows: if $BA\equiv BC$, then
$\angle BAC\equiv\angle BCA$.
\end{thm}
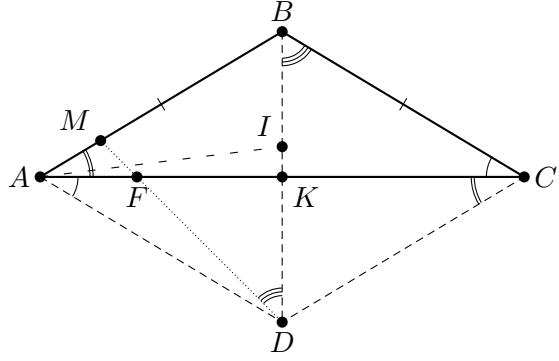
\begin{figure}[ht!]
\centering
\begin{tikzpicture}[scale=1.6,thin]
\coordinate[label=left:$A$] (A) at (-2,0);
\coordinate[label=above:$B$] (B) at (0,1.2);
\coordinate[label=right:$C$] (C) at (2,0);
\coordinate[label=below:$D$] (D) at (0,-1.2);
\coordinate[label=below right:$K$] (K) at (0,0);
\coordinate[label=below:$F$] (F) at (-1.2,0);
\coordinate[label=above left:$M$] (M) at (-1.5,0.3);

\draw[thick] (A)--(B)--(C)--cycle;
\draw[densely dashed] (B)--(D);
\draw[densely dashed] (A)--(D)--(C);
\draw[densely dotted] (M)--(D);

\fill (A) circle (1.3pt);
\fill (B) circle (1.3pt);
\fill (C) circle (1.3pt);
\fill (D) circle (1.3pt);
\fill (K) circle (1.3pt);
\fill (F) circle (1.3pt);
\fill (M) circle (1.3pt);

\pic[draw,angle radius=0.5cm] {angle=D--A--C};
\pic[draw,angle radius=0.5cm] {angle=B--C--A};
\pic[draw,angle radius=0.35cm] {angle=D--B--C};
\pic[draw,angle radius=0.40cm] {angle=D--B--C};
\pic[draw,angle radius=0.45cm] {angle=D--B--C};
\pic[draw,angle radius=0.35cm] {angle=B--D--F};
\pic[draw,angle radius=0.40cm] {angle=B--D--F};
\pic[draw,angle radius=0.45cm] {angle=B--D--F};
\pic[draw,angle radius=0.70cm] {angle=A--C--D};
\pic[draw,angle radius=0.66cm] {angle=A--C--D};
\pic[draw,angle radius=0.70cm] {angle=C--A--B};
\pic[draw,angle radius=0.66cm] {angle=C--A--B};

\PerpMark{A}{B}{0.5}{0.05}{thin}
\PerpMark{B}{C}{0.5}{0.05}{thin}

\coordinate[label=above left:$I$] (I) at (0,+0.25);
\fill (I) circle (1.3pt);
\draw[loosely dashed] (A)--(I);
\end{tikzpicture}
\caption{Theorem~\ref{PEN2_PA} Schematic.}
\label{PEN2_PA_f}
\end{figure}
\begin{proof}
\emph{(1) Midpoint $K$ and its reflection $D$.} By [\textbf{AB}], the bisector of
$\angle ABC$ exists; by \cite[Crossbar Theorem, p.~82]{Greenberg1993}, let $K\in\mathrm{int}(AC)$
be the point where it meets $AC$, so that $\overrightarrow{BK}$ bisects $\angle ABC$. By
[\textbf{C}1], let $D$ be the unique point such that $K$ is the midpoint of $BD$.

\emph{(2) Auxiliary point $F$.} By [\textbf{C}4] and [\textbf{C}1], let $F$ be the point, on the same
side of $\overleftrightarrow{BD}$ as $A$, such that $\angle BDF\equiv\angle DBC$ and
$DF\equiv BC$.

\emph{(3) Applying the crossing-angles lemma (Theorem~\ref{PEN2_CA}).} Since $K\in\mathrm{int}(AC)$
(step 1, via Crossbar), $\overleftrightarrow{BD}$ meets segment $AC$ at the interior point
$K$, so $A$ and $C$ are on opposite sides of $\overleftrightarrow{BD}$. Since $F$ and $A$ are
on the same side of $\overleftrightarrow{BD}$ (step 2), by
\cite[Corollary~(iii), p.~77]{Greenberg1993}, $F$ and $C$ are on opposite sides of
$\overleftrightarrow{BD}$. The configuration $(B,D,F,C)$ with respect to
$\overleftrightarrow{BD}$ thus satisfies the hypotheses of Theorem~\ref{PEN2_CA}
($\angle DBC\equiv\angle BDF$, $DF\equiv BC$).

By Theorem~\ref{PEN2_CA} applied to this configuration (with the correspondence
$B\leftrightarrow A$, $D\leftrightarrow B$, $F\leftrightarrow C$, $C\leftrightarrow D$ to its
own statement), there exists a point $G=\overleftrightarrow{BD}\cap FC$ with $B-G-D$, such
that $G$ is the midpoint of both $BD$ and $FC$, and $\triangle FDG\equiv\triangle CBG$.

Since $K$ is also the midpoint of $BD$ (step 1): $BK\equiv KD$ with $B-K-D$, and
$BG\equiv GD$ with $B-G-D$, so both $K$ and $G$ lie on ray $\overrightarrow{BD}$ at the same
distance from $B$ (half of $BD$); by [\textbf{C}1] (uniqueness of the point on a ray at a
given distance), $G=K$. Hence
\begin{equation*}
F-K-C\qquad\text{and}\qquad\triangle FDK\equiv\triangle CBK.
\end{equation*}

\emph{(4) $F=A$, by contradiction.} Suppose $F\neq A$. Since $K\in AC$ and $F-K-C$, $F$ lies
on $\overleftrightarrow{AC}$, so $A,F,K,C$ are collinear. Suppose, for contradiction,
$A-K-F$. Then segment $AF$ meets $\overleftrightarrow{BD}$ at $K$, so $A$ and $F$ would be on
opposite sides of $\overleftrightarrow{BD}$, contradicting step (2). Hence $A-K-F$ is excluded,
and by [\textbf{O}3] applied to the three distinct collinear
points $A,F,K$, either $A-F-K$ or $F-A-K$ holds. Hence either $A-F-K-C$ or $F-A-K-C$.

The configuration preceding the construction of $F$ -- namely $K$, $D$, and the hypothesis
$BA\equiv BC$ -- is symmetric under exchanging $A$ and $C$. The case $F-A-K-C$ therefore
reduces to the case $A-F-K-C$ by relabeling $A\leftrightarrow C$ (with $F$ replaced by the
point constructed via step (2) with $A$ in place of $C$). We may thus assume WLOG
$A-F-K-C$.

Since $A-F-K$, ray $\overrightarrow{DF}$ lies in the interior of $\angle ADB$ (as $B-K-D$
gives $\overrightarrow{DK}=\overrightarrow{DB}$). By Crossbar \cite[p.~82]{Greenberg1993},
$\overrightarrow{DF}$ meets segment $AB$ at a point $M$ such that $A-M-B$.

By construction $\angle MDB\equiv\angle MBD$ in $\triangle DMB$ by Lemma~\ref{PEN2_PA_CONV}, $MD\equiv MB$.

But $M-F-D$ gives $MD>FD$, and $A-M-B$ gives $AB>MB$. So
\begin{equation*}
MD>FD\equiv BC\equiv AB>MB\equiv MD,
\end{equation*}
a contradiction. Hence $F=A$.

\emph{(5) Second reflection point $F'$.} Since $F=A$, step (3) also gives $AK\equiv CK$. By
[\textbf{C}1], let $F'$ be the unique point on ray $\overrightarrow{KC}$ such that
$KF'\equiv KA$; then $KF'\equiv KA\equiv KC$, and since $F'$ lies on ray
$\overrightarrow{KC}$ at this distance, uniqueness (again [\textbf{C}1]) gives $F'\equiv C$.

\emph{(6) Symmetry between $\triangle ABC$ and $\triangle CDA$.} Replicating, starting from
$\triangle CDA$, the same construction that led from $\triangle ABC$ to $\triangle CDA$, we
obtain $\triangle ABC$ again. Both triangles are isosceles; $BD$ is the bisector of the
angles at $B$ and at $D$, and the corresponding half-angles are congruent. Moreover
$AC\equiv AC$, $AB\equiv CD$, $BC\equiv DA$, and $\angle BCA\equiv\angle CAD$,
$\angle BAC\equiv\angle ACD$.

\emph{(7) Second bisector and conclusion.} Note that $\angle ABD=\angle ABK$ and
$\angle ADB=\angle ADK$ (ray identities, $K$ being on $BD$), so the congruent half-angles of
step (6) already give $\angle ABD\equiv\angle ADB$.

By [\textbf{AB}], let $I\in\mathrm{int}(BD)$ be such that $\overrightarrow{AI}$ bisects
$\angle DAB$. Consider $\triangle AIB$ and $\triangle AID$: they share side $AI$, have
$\angle IAB\equiv\angle IAD$ (bisector, by construction), and, since
$\overrightarrow{BI}=\overrightarrow{BD}$ and $\overrightarrow{DI}=\overrightarrow{DB}$,
$\angle ABI=\angle ABD\equiv\angle ADB=\angle ADI$ (just shown). By SAA -- side $AI$, angle
at $A$, angle at $B$/$D$ -- $\triangle AIB\equiv\triangle AID$, so $BI\equiv DI$: $I$ is the
midpoint of $BD$. Since $BI\equiv ID$ and $BK\equiv KD$, with $B-I-D$ and $B-K-D$, both $I$
and $K$ lie on ray $\overrightarrow{BD}$ at the same distance from $B$; by [\textbf{C}1]
(uniqueness of the point on a ray at a given distance), $I=K$. Hence $\overrightarrow{AK}$
bisects $\angle DAB$:
\begin{equation*}
\angle DAK\equiv\angle BAK,\qquad\text{i.e.}\qquad \angle DAC\equiv\angle BAC.
\end{equation*}

Combined with $\angle BCA\equiv\angle DAC$ (step 6), this gives
\begin{equation*}
\angle BAC\equiv\angle DAC\equiv\angle BCA,
\end{equation*}
which is [\textbf{PA}].
\end{proof}

\section{Consequence: an improved reconstruction}\label{PEN2_s4}

\begin{cor}\label{PEN2_MAIN}
In $\MG^{-}$,
\begin{equation*}
\textrm{SAA},\ [\textbf{RCT}],\ [\textbf{AB}],\ [\textbf{SA}]\;\vdash\;\textrm{SAS}.
\end{equation*}
\end{cor}
\begin{proof}
By Theorem~\ref{PEN2_PA}, [\textbf{PA}] follows from SAA, [\textbf{RCT}], and [\textbf{AB}]. Case
(ii) of Theorem 3.11 of \cite[Section~3]{Volpe2026I} then applies with
[\textbf{PA}] discharged
as a derived lemma rather than assumed as a hypothesis.
\end{proof}

This improves the deductive equivalence \eqref{PEN2_ESAA_OLD}: [\textbf{PA}] is no longer
needed as an independent hypothesis, and the equivalence of \cite{Volpe2026I} accordingly
strengthens to
\begin{equation}\label{PEN2_ESAA_NEW}
\textrm{SAS}\;\dashv\vdash\;\textrm{SAA},\ [\textbf{RCT}],\ [\textbf{AB}],\ [\textbf{SA}].
\end{equation}

\section{Concluding remarks}\label{PEN2_s5}

The result of this paper closes one of the two questions left open in
\cite{Volpe2026I}: [\textbf{PA}] is not an independent requirement alongside SAA,
[\textbf{RCT}], [\textbf{AB}], and [\textbf{SA}] in our reconstruction of SAS -- it is a
consequence of the first three. The corresponding question on the SSS branch is taken up
below.

A further consequence of eliminating [\textbf{PA}] concerns the security of the SAA
reconstruction itself. The concluding remarks of \cite{Volpe2026I} noted that, among the
three reconstructions, only ASA could be shown to rest on ``a distinctly more secure
footing'': since $\mathbb{E}_{H}^{2}$ (the model of \cite[Section~4]{Volpe2026I}) satisfies
[\textbf{RCT}] but not ASA, the soundness
theorem gives $[\textbf{RCT}]\not\vdash\textrm{ASA}$ in $\MG^{-}$, so the reconstruction
genuinely requires ASA, not merely [\textbf{RCT}] in disguise. No analogous certainty was
available for SAA, because $\mathbb{E}_{H}^{2}$ fails [\textbf{PA}] -- one of the SAA
reconstruction's own hypotheses -- and so could not be used to test whether SAA itself is
required beyond its auxiliary principles.

With [\textbf{PA}] eliminated by Theorem~\ref{PEN2_PA}, this obstruction disappears: the
hypothesis set of Corollary~\ref{PEN2_MAIN} reduces to SAA, [\textbf{RCT}], [\textbf{AB}],
[\textbf{SA}], and $\mathbb{E}_{H}^{2}$ satisfies all three auxiliary members while failing
SAA itself. By the same soundness argument used for ASA,
$[\textbf{RCT}],[\textbf{AB}],[\textbf{SA}]\not\vdash\textrm{SAA}$ in $\MG^{-}$: the
reconstruction of Corollary~\ref{PEN2_MAIN} genuinely requires SAA, not merely its auxiliary
principles in disguise. The SAA reconstruction now stands on exactly the same secure footing
as the ASA reconstruction.

This raises the natural question of whether the SSS reconstruction could be brought to the
same footing by an analogous elimination of [\textbf{HA}]. Unlike [\textbf{PA}],
[\textbf{HA}] does not appear to be a consequence of the principles already assumed alongside
it (SSS, [\textbf{RCT}], [\textbf{MS}]); no argument of the kind given in
Section~\ref{PEN2_s3} for [\textbf{PA}] is available to us here. The relevant question is
therefore not whether [\textbf{HA}] is redundant among the current hypotheses, but whether it
could be \emph{replaced} by some as-yet-unidentified principle $X$ -- provable in $\MG$, and, crucially,
satisfied by $\mathbb{E}_{H}^{2}$ -- such that SSS, [\textbf{RCT}], [\textbf{MS}], $X$ still
suffices to derive SAS. Since $\mathbb{E}_{H}^{2}$ fails SSS (and SAS) outright, any such $X$
would, by the same soundness argument used above, certify that
[\textbf{RCT}], [\textbf{MS}], $X\not\vdash$ SSS in $\MG^{-}$, placing the SSS reconstruction
on the same secure footing as the other two. We have not identified such an $X$, and we do
not know whether one exists; we record the question here as a direction for future work.

We do not believe [\textbf{SA}] can be eliminated in turn: it enters our reconstruction
through Theorem 3.2 of \cite[Section~3]{Volpe2026I} (the congruence of all right
angles), which is
what makes [\textbf{HA}] applicable to right triangles with distinct vertices -- as it is in
the proof of the main theorem, where the feet $P$ and $Q$ of the two perpendiculars belong to
different triangles. Without [\textbf{SA}], this comparison is not available, and we see no
route around it.

Finally, we have not investigated in this paper whether the converse substitution is
possible, that is, whether [\textbf{RCT}], [\textbf{SA}], and [\textbf{AB}] together suffice
to derive the angle addition axiom N2 of \cite{Donnelly2025}. If they did, [\textbf{RCT}] and
N2 would be capable of establishing the same results on this shared basis, rather than
merely overlapping in consequences, and the asymmetry we have described in
Section~\ref{PEN2_s2} -- that the reconstruction of \cite{Donnelly2025} can borrow ours, but
not evidently the reverse -- would need to be reconsidered.


\end{document}